\documentclass[sn-mathphys]{sn-jnl}% Math and Physical Sciences  
\usepackage{enumerate}
\theoremstyle{thmstyleone}%
\newtheorem{theorem}{Theorem}[section]%  meant for continuous numbers  
\newtheorem{lemma}[theorem]{Lemma}

\newtheorem{problem}[theorem]{Problem} 
\newtheorem{hypothesis}[theorem]{Hypothesis} 

\theoremstyle{thmstyletwo}%
\newtheorem{example}{Example}% 

\theoremstyle{thmstylethree}%
\newtheorem*{remark}{Remark}%

\renewcommand{\A}{\mathrm{A}} \newcommand{\AGL}{\mathrm{AGL}}  \newcommand{\Alt}{\mathrm{Alt}} \newcommand{\Aut}{\mathrm{Aut}}
  
 \newcommand{\Cen}{\mathbf{C}}        
\newcommand{\D}{\mathrm{D}}

\newcommand{\M}{\mathrm{M}} \newcommand{\magma}{\textsc{Magma}}

\newcommand{\Out}{\mathrm{Out}}
  \newcommand{\PGL}{\mathrm{PGL}} \newcommand{\PGaL}{\mathrm{P\Gamma L}}     \newcommand{\PSL}{\mathrm{PSL}}

 \newcommand{\SL}{\mathrm{SL}}  \newcommand{\Soc}{\mathrm{Soc}}     \newcommand{\Sy}{\mathrm{S}} \newcommand{\Sym}{\mathrm{Sym}} 

 \newcommand{\ZZ}{\mathrm{C}}

\renewcommand{\a}{\alpha}

\numberwithin{equation}{section}  
\newcommand{\Fix}{\mathrm{Fix}}

\begin{document}

\title{Block-transitive $3$-$(v,k,1)$ designs associated  with alternating groups}

%\author[Lan]{Ting Lan}
%\address{Ting Lan\\
%School of Mathematics and Statistics\\
%Central South University\\
%Changsha 410083, Hunan, P.R. China.}
%\email{lanting0603@163.com}
%
%
%\author[Liu]{Weijun Liu}
%\address{Weijun Liu \\
%School of Mathematics and Statistics\\
%Central South University\\
%Changsha 410083, Hunan, P.R. China.}
%\email{wjliu6210@126.com}
%
%\author[Yin]{Fu-Gang Yin}
%\address{Fu-Gang Yin\\
%School of Mathematics and Statistics\\
%Central South University\\
%Changsha 410083, Hunan, P.R. China.}
%\email{18118010@bjtu.edu.cn}
%

\author[1]{\fnm{Ting} \sur{Lan}}\email{lanting0603@163.com}
\equalcont{These authors contributed equally to this work.}

\author[2,1]{\fnm{Weijun} \sur{Liu}}\email{wjliu6210@126.com}
\equalcont{These authors contributed equally to this work.}

\author*[1]{\fnm{Fu-Gang} \sur{Yin}}\email{18118010@bjtu.edu.cn}
\equalcont{These authors contributed equally to this work.}

\affil[1]{\orgdiv{School of Mathematics and Statistics}, \orgname{Central South University}, \orgaddress{\street{Yuelu South Road}, \city{Changsha}, \postcode{410083}, \state{Hunan}, \country{P.R. China}}}
 
\affil[2]{\orgdiv{College of General Education}, \orgname{Guangdong University of Science and Technology}, \orgaddress{\city{Dongguan}, \postcode{523083}, \state{Guangdong}, \country{P.R. China}}}
 

\abstract{
Let $\mathcal{D}$ be a  nontrivial $3$-$(v,k,1)$ design admitting a  block-transitive  group  $G$ of automorphisms.
A recent work of Gan and the second author asserts that $G$ is either affine or almost simple.  
In this paper, it is proved that if $G$ is almost simple with socle an alternating group, then  $\mathcal{D}$ is the unique $3$-$(10,4,1)$ design, and $G=\PGL(2,9)$, $\M_{10}$ or $\mathrm{Aut}(\mathrm{A}_6 )=\mathrm{S}_6:\mathrm{Z}_2$, and $G$ is flag-transitive.
%This combine with a result of Kantor and a work of Camina, Neumann and Praeger completes the classification of   $G$-$(v,k,1)$ design admitting a  block-transitive  group  $G$ of automorphisms. 
%Let $\mathcal{D}$ be a  nontrivial $t$-$(v,k,1)$ design admitting a  block-transitive  group  $G$ of automorphisms, where $G$ is almost simple with socle an alternating group.
%A result of Kantor suggests that $t<4$, and the case $t=2$ has been settled by Camina, Neumann and Praeger. 
%In this paper, it is proved that if $t=3$ then $G$ is the  automorphism group of $\A_6$, and $\mathcal{D}$ is $G$-flag-transitive with parameters $v=10$ and $k=4$. 
}




\maketitle

\section{Introduction}\label{sec1}
A $t$-$(v,k,\lambda)$ design  $\mathcal{D}= (\mathcal{P},\mathcal{B})$  is an incidence structure consisting of a set $\mathcal{P}$ of $v$  \emph{points} and a set $\mathcal{B}$ of  $k$-subsets of $\mathcal{P}$  called \emph{blocks}  such that, each block in $\mathcal{B}$ has size $k$, and each $t$-subset of $\mathcal{P}$ lies in exactly $\lambda$ blocks from $\mathcal{B}$.
Design $\mathcal{D}$ is said to be \emph{trivial} if $\mathcal{B}$ consists of all the $k$-subsets of $\mathcal{P}$.
A flag of $\mathcal{D}$ is a  pair $(\alpha,B)$ where $\alpha $ is a point and $B$ is a block containing $\alpha$.
An automorphism of $\mathcal{D}$ is a permutation on $\mathcal{P}$ which permutes the blocks among themselves.
For a subgroup $G$ of the automorphism group $\Aut(\mathcal{D})$ of $\mathcal{D}$, the design $\mathcal{D}$ is said to be \emph{$G$-block-transitive} if $G$ acts transitively on the set of blocks, and is said to be \emph{block-transitive} if it is $\Aut(\mathcal{D})$-block-transitive.
The point-transitivity and flag-transitivity are defined similarly.
Clearly, flag-transitivity infers block-transitivity, and by a theorem of Block (see~\cite{Block1967}) block-transitivity implies  point-transitivity.

For a nontrivial $G$-block-transitive $t$-$(v,k,\lambda)$ design, it is proved by Cameron and Praeger~\cite{CP1993lagert}  that $G$ is $\lfloor t/2 \rfloor$-homogeneous on the point set and  $t\leq 7$, and while if $G$  is flag-transitive, then $G$ is $\lfloor (t+1)/2 \rfloor$-homogeneous and $t\leq 6$.
Furthermore, Cameron and Praeger~\cite{CP1993lagert} conjectured that there exists no nontrivial block-transitive $6$-$(v,k,\lambda)$ designs.
%Note that a $7$-$(v,k,\lambda)$ design is a $6$-$(v,k,\lambda')$ for some $\lambda'$.
Huber~\cite{H2010t=7} confirmed the nonexistence of nontrivial block-transitive $7$-$(v,k,1)$ designs, and then in~\cite{H2010} he proved that apart from $G=\mathrm{P\Gamma L}_2(p^e)$ with $p\in \{2,3\}$ and $e$ an odd prime power, there exists no nontrivial $G$-block-transitive $6$-$(v,k,1)$ designs. 
The exceptional case was studied by Tan, Liu and Chen~\cite{TLC2014} for the case $k\leq 10000$.
 
%In this paper, we focus on the $t$-$(v,k,\lambda)$ designs, which is also called \emph{Steiner $t$-designs}.
 A $t$-$(v,k,\lambda)$ design with $\lambda=1$ is called a \emph{Steiner $t$-design}, and a $2$-$(v,k,1)$ design is also called a \emph{linear space}. 
For a $t$-$(v,k,1)$ design, we need
$t\geq 2$ and $t<k<v$ to avoid trivial examples, and so a $t$-$(v,k,1)$ design is said to be \emph{nontrivial} if $t\geq 2$ and $t<k<v$.  
There have been a great deal of efforts to classify block-transitive $2$-$(v,k,1)$ designs in the past fifty years.
For example, point $2$-transitive $2$-$(v,k,1)$ designs are classified by Kantor~\cite{K1985}, and flag-transitive $2$-$(v,k,1)$ designs apart from those with one-dimensional affine automorphism groups are classified by  Buekenhout et. al~\cite{BDDKLS1990}.
In 2001, Camina and Praeger~\cite{CP2001} proved that
if a $2$-$(v,k,1)$ design  is $G$-block-transitive and $G$-point-quasiprimitive, then $G$ is either an affine group or an almost simple group.  
This result has inspired the study of $G$-block-transitive $2$-$(v,k,1)$ designs  with $G$ an almost simple group, such as 
alternating groups~\cite{CNP2003},  sporadic simple groups~\cite{CS2000},  simple groups of Lie type of  small ranks~\cite{G2007, L2001, L2003, L2003a, L2003b,LLG2006, LLM2001, LZLF2004,   Z2002, Z2005,ZLL2000}, and large dimensional classical groups~\cite{CGZ2008}. 
%and the complete classification results include alternating groups~\cite{CNP2003}, sporadic simple groups~\cite{CS2000},  two-dimensional projective linear groups~\cite{L2003}, Ree groups ${}^2\G_2(q)$~\cite{LZLF2004}, Steinberg groups ${}^3\D_4(q)$~\cite{LDG2006} 

Compared with $2$-$(v,k,1)$ designs, the results for block-transitive $t$-$(v,k,1)$ designs with $t \in \{3,4,5\}$ are rare.
The flag-transitive  $t$-$(v,k,1)$ designs for $t\in \{3,4,5\}$ has been classified by Huber~\cite{H2005,H2007-4flag,H2009-5flag}.
Note that $1$-homogenous groups are simply transitive groups, while $2$-homogenous groups are almost  $2$-transitive (with only one affine group as exception, see~\cite{K1972}), and in particular, $2$-homogenous groups are either affine or almost simple. 
Consequently, for a  $G$-block-transitive $t$-$(v,k,1)$ design with $t\in \{4,5\}$, $G$ is either affine or almost simple.
Huber~\cite{H2010affine} proved that if $G$ is an affine group, then there exists no $G$-block-transitive  nontrivial $t$-$(v,k,1)$ design for $t\in \{4,5\}$  except the one-dimensional affine case. 
Thus the study of $G$-block-transitive $t$-$(v,k,1)$ design with $t\in \{4,5\}$ has been essentially reduced to the case that $G$ is an almost simple $2$-transitive group.
Based on the study of Cameron and Praeger~\cite{CP1993} on block-transitive and point-imprimitive $t$-$(v,k,\lambda)$ designs, it is proved by Mann and Tuan~\cite[Corollary~2.3(a)]{MT2001} that there exists no block-transitive and point-imprimitive $3$-$(v,k,1)$ design. 
Very recently, it is shown by Gan and the second author~\cite{GL2022+}  that for a nontrivial $G$-block-transitive $3$-$(v,k,1)$ design, the group $G$ is either affine or almost simple. 
This suggests the following problem. 

\begin{problem}\label{problem}
Classify nontrivial $G$-block-transitive $3$-$(v,k,1)$ designs, where $G$ is an almost simple group.
\end{problem}
 
Those $\PSL_n(q)$-block-transitive $3$-$(v,k,1)$ designs with $v=(q^n-1)/(q-1)$ are determined by Tang, Liu and Wang~\cite{TLW2013}. This paper is devoted to solve Problem~\ref{problem} in the case where the socle of $G$ is an alternating group. 
Note that $\Aut(\A_6) =\mathrm{S}_6:\mathrm{Z}_2$, and following Atlas~\cite{Atlas},  $\Aut(\A_6) $ has three subgroups of index $2$, namely $ \A_6.2_i$ with $i\in \{1,2,3\}$.
Moreover, $  \A_6.2_1 \cong\mathrm{P\Sigma L}_2(9) \cong \Sy_6$, $ \A_6.2_2\cong \PGL(2,9)$ and $ \A_6.2_3 \cong\M_{10}$ ($\M_{10}$ is the stabilizer of the Mathieu group $\M_{11}$ on its natural action of degree $11$).

\begin{theorem}\label{th:main}
Let $G$ be an almost simple group with alternating socle $\mathrm{A}_n(n\geq 5)$.
Suppose that  $\mathcal{D} $ is a nontrivial $G$-block-transitive $3$-$(v,k,1)$ design. 
Then $G=\mathrm{PGL}_2(9)$, $\mathrm{M}_{10}$ or $\mathrm{S}_6:\mathrm{Z}_2$,  and  $\mathcal{D}$ is $G$-flag-transitive with parameters $v=10$ and $k=4$.  
\end{theorem}

The $3$-$(10,4,1)$ design appearing in Theorem~\ref{th:main} is constructed in Example~\ref{ex} in Section~3, and it is a special case of a family of flag-transitive $3$-$(v,k,1)$ designs on $\PSL_2(q)$ (noticing that $\PSL_2(9)\cong \A_6$), which is  introduced in~\cite[Theorem~(2)]{H2005} and~\cite[Theorem~3(b)]{K1985}.

Let  $\mathcal{D} $ be a nontrivial $G$-block-transitive $t$-$(v,k,1)$ design, where $G$ is  an almost simple group  with alternating socle. 
According to Kantor~\cite[Theorem~3]{K1985}, there exists no such design $\mathcal{D} $ if $t\geq 4$.
It is proved by Camina, Neumann and Praeger~\cite{CNP2003}  that if $t=2$ then $G=\A_7$ or $\A_8$ and $\mathcal{D} $ is the $1$-skeleton of the $3$-dimensional projective geometry  over $\mathbb{F}_2$, which is $G$-flag-transitive and has parameter $v=15$ and $k=3$.
Therefore, Theorem~\ref{th:main} together with~\cite[Theorem~3]{K1985} and~\cite{CNP2003} give the complete classification of such designs.

\section{Preliminaries}
For a finite group $G$,  the socle of $G$, denoted by $\Soc(G)$, is the product of all its minimal normal subgroups, and the group $G$ is said to be \emph{almost simple} if $\Soc(G)$ is a nonabelian simple group. 
For an element $g$ of $G$, let $\Cen_G(g)$ denote the centralizer of $g$ in $G$.
If $G$ is a permutation group on a finite set $\Omega$, then we use $\Fix_{\Omega}(g)$ to denote the set of points in $\Omega$ fixed by $g$.
For a point $\alpha \in \Omega$, denote by $G_\alpha$ the stabilizer of $\alpha$ in $G$, and by $\alpha^G$ the orbit of $G$ containing $\alpha$.
The alternating group and the symmetric group on a finite set  $\Omega$ is  denoted by  $\Alt(\Omega)$ and $\Sym(\Omega)$, respectively. 
For notations of groups, we follow Atlas~\cite{Atlas}. 
In particular, for two finite groups $H$ and $K$, we use $H\,{:}\, K$ to denote the semidirect product of $H$ by $K$.
For a real number $x$, let $\lfloor x \rfloor$ be the   largest integer less than or equal to $x$.

\subsection{$3$-$(v,k,1)$ designs}
In this subsection, we introduce some results about $3$-$(v,k,1)$ designs. 
The following fact is well known, and we refer the readers to~\cite[p.180]{DM-book}.
 \begin{lemma}\label{lm:brlambda2}
Let $\mathcal{D}=(\mathcal{P},\mathcal{B})$ be a $3$-$(v,k,1)$ design.  
\begin{enumerate}[\rm (a)]
\item ~$\displaystyle \vert \mathcal{B}\vert=\frac{v(v-1)(v-2)}{k(k-1)(k-2)}$. 
 
%\item The parameter $k$ is the sum of lengths of orbits of $G_\alpha$ acting on a block containing $\alpha$.

\item For every point $\alpha$ in $\mathcal{P}$, the number of blocks containing $\alpha$ is 
\[
\lambda_1:=\frac{(v-1)(v-2)}{(k-1)(k-2)}. 
\]
%Moreover, $\lambda_2$ is the sum of lengths of orbits of $G_\alpha$ acting on the set of blocks containing $\alpha$.

  \item For two distinct points  $\alpha$ and $\beta$ in $\mathcal{P}$, the number of blocks containing $\alpha$ and $\beta$ is 
 \[
\lambda_2:=\frac{ v-2 }{ k-2 }. 
\]

\end{enumerate} 
\end{lemma}

The parameters $v$ and $k$ have an important relation.

\begin{lemma}[{\cite{C1976}}]\label{lm:vk}
Let $\mathcal{D} $ be a $3$-$(v,k,1)$ design. Then 
\[
v\geq k^2-3k+4.
\]
In particular, $k\leq \lfloor \sqrt{v}\rfloor+2$.
\end{lemma}

The idea of the following lemma is original from~\cite[Proposition~2.7]{CNP2003}.
\begin{lemma}\label{lm:vkGCGg}
Let $\mathcal{D}=(\mathcal{P},\mathcal{B})$ be a $G$-block-transitive $3$-$(v,k,1)$ design.  
Suppose that $G$ has an element $g \neq 1$ such that  $\langle g \rangle$ has an orbit on  $\mathcal{P}$ with length at least $3$.  Then
\[
\frac{(v-1)(v-2)}{k(k-1)(k-2)} \leq \frac{\vert G\vert}{\vert\Cen_G(g)\vert}.
\]
\end{lemma}

\begin{proof}
Let $\Omega$ be an orbit of $\langle g \rangle$ with length at least $3$. 
Take three distinct points $\alpha,\beta,\gamma \in \Omega$ such that $\beta=\alpha^g$ and $\gamma=\alpha^{g^2}$. 
Let $X=G_\alpha$ and let $U=G_\alpha \cap G_\beta \cap G_\gamma$. Note that
\[
\alpha^{g^i\Cen_{G_\alpha}(g)}=\alpha^{\Cen_{G_\alpha}(g)g^i}=\alpha^{g^i} \text{ for all } i \in \{1,2\}. 
\]
Therefore, $\Cen_{G_\alpha}(g)$ fixes $\alpha^g=\beta$ and  $\alpha^{g^2}=\gamma$, which implies that $U \geq \Cen_{G_\alpha}(g)=X \cap \Cen_G(g)$. 
Then we have 
\[
\frac{\vert X\vert}{\vert U \vert} \leq \frac{\vert X\vert }{\vert X \cap \Cen_G(g)\vert }= \frac{\vert X\Cen_G(g)\vert }{\vert \Cen_G(g)\vert } \leq  \frac{\vert G\vert }{\vert \Cen_G(g)\vert }.\]
Let $B$ be the unique 
block containing the $3$-subset $\{ \alpha,\beta,\gamma  \}$.
Then $U \leq G_{B}$.
Hence
\[
\vert \mathcal{B}\vert =\frac{\vert G\vert }{\vert G_B\vert } \leq \frac{\vert G\vert }{\vert U\vert }=\frac{\vert G\vert }{\vert X\vert }\frac{\vert X\vert }{\vert U\vert }\leq v\frac{\vert G\vert }{\vert \Cen_{G}(g)\vert }.
\] 
This together with  Lemma~\ref{lm:brlambda2}(a) prove the lemma. 
\end{proof}

Next we collect several important properties of block-transitive $3$-$(v,k,1)$ designs from~\cite{GL2022+}.

\begin{lemma}[{\cite[Lemma~2.5]{GL2022+}}]\label{lm:gfixeno}
Let $\mathcal{D}=(\mathcal{P},\mathcal{B})$ be a $G$-block-transitive $3$-$(v,k,1)$ design. 
If there exists an element $g$ of $G$ with order $3$ and $g$ fixes no points, then $k$ divides $v$.
\end{lemma}

\begin{lemma}[{\cite[Lemma~2.7]{GL2022+}}] \label{lm:vdivisiblebyk4}
Let $\mathcal{D}=(\mathcal{P},\mathcal{B})$ be a $G$-block-transitive $3$-$(v,k,1)$ design. 
If   $v$  is divisible by $k$ and $4$, then $\mathcal{D}$ is $G$-flag-transitive.
\end{lemma}

Let $G$ be a transitive permutation group on a set $\Omega$ and let $\alpha \in \Omega$. 
A  $G_\alpha$-orbit on $\Omega$ is called a \emph{suborbit} of $G$ relative to $\alpha$, and the length of a $G_\alpha$-orbit on $\Omega$ is called a \emph{subdegree} of $G$. 
Clearly, $G_\alpha$ has a trivial orbit $\{\alpha\}$.
A subdegree  is said to be \emph{nontrivial} if the corresponding suborbit is not $\{\alpha\}$. 

\begin{lemma}[{\cite[Lemma~2.8]{GL2022+}}]\label{lm:vkGad}
Let $\mathcal{D}=(\mathcal{P},\mathcal{B})$ be a $G$-block-transitive $3$-$(v,k,1)$ design. 
 Then
\begin{enumerate}[\rm (a)]
\item $(v-1)(v-2) $ divides $ k(k-1)(k-2)\vert G_\alpha\vert $ for every $\alpha \in \mathcal{P}$.
\item $(v-1)(v-2) $ divides $ k(k-1)(k-2)d(d-1) $ for every nontrivial subdegree $d$ of $G$ on $\mathcal{P}$.
\end{enumerate} 
\end{lemma}

\begin{lemma}[{\cite[Lemma~3.2]{GL2022+}}]\label{lm:FixHvk}
Let $\mathcal{D}=(\mathcal{P},\mathcal{B})$ be a $G$-block-transitive $3$-$(v,k,1)$ design. 
Suppose that $1 \neq H\leq G$ has an orbit on  $\mathcal{P}$ with length at least $3$. 
Then
\[
\vert \Fix_{\mathcal{P}}(H)\vert  \leq  \frac{2(v-k)}{k-2}+k-2.
\]
\end{lemma}

We end this subsection with an observation which is useful in computation.
\begin{lemma}\label{lm:aninequation}
Let $v,k,c$ be positive integers such that $3<k<v$ and 
\[v\geq k^2-3k+4, \text{ and } (v-1)(v-2)\leq ck(k-1)(k-2).\]
Then $v< (c+2)^2$.
\end{lemma}

\begin{proof}
Since $v\geq k^2-3k+4$, we have $v\geq (k-2)^2$, and $v-2\geq k^2-3k+2=(k-1)(k-2)$. 
It follows that
\[ 
\sqrt{v}-2< \frac{(\sqrt{v}-2)(v-1)}{v-4}=\frac{ v-1 }{\sqrt{v}+2} \leq \frac{v-1}{k } \leq  \frac{v-1}{k }\frac{v-2}{(k-1)(k-2)}\leq c.
\] 
Hence,  $v< (c+2)^2$.
\end{proof}

\subsection{Subdegrees of two classes of permutation groups}
In this subsection, we study the subdegrees of two classes of permutation groups associated with the alternating groups and symmetric groups. 
The first one is well-known.

\begin{lemma}\label{lm:subdegree}
Let $m$ and $n$ be positive integers with $n\geq 5$ and $n\geq 2m+1$, and let $\Omega$ be the set of $m$-subsets of $\{1,2,\dots,n\}$. 
Let $G=\Sy_{n}$ and $L=\A_{n}$ acting  naturally on $\Omega$. 
Then both $G$ and $L$ have $m+1$ suborbits on $\Omega$, and have the same  nontrivial subdegrees:  
\[d_i=\binom{m}{i}\binom{n-m}{m-i}=\frac{m!(n-m)!}{i!(m-i)!^2(n+i-2m)!} \text{ for all }i\in \{0,1,\dots,  m-1\}.\] 
\end{lemma} 

\begin{proof}
Fix $\alpha \in \Omega$. For every $i\in \{0,1,\dots,m\}$, let 
\[\Delta_i=\{ \beta \in \Omega \mid  \vert \alpha \cap \beta \vert =i \}.\]
Clearly, $\Delta_m=\{ \alpha\}$, and it is the trivial suborbit of $G$ with respect to $\alpha$.
Note that 
\[ G_\a=\Sym(\alpha) \times \Sym(\overline{\alpha}) \text{ with } \overline{\alpha}=\{1,2,...,n \}\setminus \alpha. \]

Let $i\in \{0,1,\dots,m-1\}$. Take $\beta,\gamma$ from $\Delta_i$ and let $\beta_1=\beta \cap \alpha$, $\beta_2=\beta \cap \overline{\alpha}$, $\gamma_1=\gamma \cap \alpha$ and $\gamma_2=\gamma \cap \overline{\alpha}$.
Since $ \vert \beta_1 \vert = \vert \gamma_1 \vert =i$ and $ \vert \beta_2 \vert = \vert \gamma_2 \vert =m-i$,
one can find permutations $g_1 \in \Sym(\alpha)$ and $g_2 \in \Sym(\overline{\alpha})$ such that $g_1$ swaps $\beta_1$ and $\gamma_1$, and $g_2$ swaps $\beta_2$ and $\gamma_2$.
Consequently, $g_1g_2$ swaps $\beta$ and $\gamma$.
This implies that  $\beta$ and $\gamma$ are in the same orbit of $G_\alpha$.
The arbitrariness of the choices of $\beta$ and $\gamma$ implies that $\Delta_i$ is contained in a $G_\alpha$-orbit.
Furthermore, for every $g \in G_\alpha$  and for every $\beta \in \Delta_i$, since 
\[ \vert \alpha \cap \beta^g \vert = \vert \alpha^g \cap \beta^g \vert = \vert (\alpha \cap \beta)^g \vert = \vert \alpha \cap \beta \vert =i, \]
we conclude that $\Delta_i$ is exactly a $G_\alpha$-orbit.
Therefore, $G$ has $m$ nontrivial suborbits $\Delta_i$ for all $i\in \{0,1,...,m-1\}$ on $\Omega$, and the nontrivial subdegrees are 
\[d_i= \vert \Delta_i \vert =\binom{m}{i}\binom{n-m}{m-i}=\frac{m!(n-m)!}{i!(m-i)!^2(n+i-2m)!}.\] 

Now we consider the action of $L$ on $\Omega$. 
Note that for every $\beta \in \Delta_i$ with $i\in \{0,1,...,m-1\}$, writing $\overline{\beta}=\{1,2,\dots,n\} \setminus \beta$, there holds that
\[
G_{\alpha\beta}:=G_{\alpha} \cap G_{\beta}=\Sym(\alpha \cap \beta) \times \Sym(\alpha \cap \overline{\beta})  \times \Sym(\overline{\alpha}\cap \beta) \times \Sym(\overline{\alpha} \cap \overline{\beta}).
\]
Since $n\geq 5$, we conclude that both $G_\alpha$ and $G_{\alpha\beta}$ contain an odd permutation.
It follows that  $ \vert G_{\alpha} \vert / \vert L_{\alpha} \vert =2$ and $ \vert  G_{\alpha\beta} \vert / \vert  L_{\alpha\beta} \vert =2$, and hence $ \vert L_{\alpha} \vert / \vert L_{\alpha\beta} \vert = \vert G_{\alpha} \vert / \vert G_{\alpha\beta} \vert $.
This implies that both $G$ and $L$ have $m+1$ suborbits on $\Omega$, and have the same  nontrivial subdegrees.  
\end{proof}

\begin{lemma}\label{lm:subdegree-imprimitive}
Let $m\geq 3$ be a positive integer  and let $\Omega$ be  the set of partitions of $\{1,2,\dots,2m\}$ with $2$ blocks of size $m$. 
Let $G=\Sy_{2m}$ and $L=\A_{2m}$ acting naturally on $\Omega$. 
Then both $G$ and $L$ have $\lfloor m/2 \rfloor+1$ suborbits on $\Omega$, and have the same  nontrivial subdegrees: 
\[
d_i:=2^{-\lfloor \frac{2i}{m}\rfloor}\left(\frac{m!}{(m-i)!i!}\right)^2 \text{ for all }i\in \{1,2,\dots,\lfloor m/2 \rfloor\}.
\] 
\end{lemma}

\begin{proof}
Fix $\alpha=\{ V_1,V_2\}\in \Omega$.  
By~\cite[Lemma~2.9]{YFX2022+}, we conclude that  $\beta =\{ W_1,W_2\}$ and $\gamma=\{ U_1,U_2\}$ are in a same $G_\alpha$-orbit if and only if there exist permutation matrices $P$ and $Q$ such that 
\[\begin{pmatrix}
 \vert V_1 \cap U_1 \vert  &  \vert V_1 \cap U_2 \vert  \\
 \vert V_2 \cap U_1 \vert  &  \vert V_2 \cap U_2 \vert  \\
\end{pmatrix}  = P\begin{pmatrix}
 \vert V_1 \cap W_1 \vert  &  \vert V_1 \cap W_2 \vert  \\
 \vert V_2 \cap W_1 \vert  &  \vert V_2 \cap W_2 \vert  \\
\end{pmatrix} Q.  \]
Note that in  matrices $( \vert V_r \cap U_s \vert )_{2\times 2}$ and $( \vert V_r \cap W_s \vert )_{2\times 2}$, the sum of every row equals to $m$  and the sum of every column still equals to $m$. 
Therefore, the nontrivial suborbits of $G $ relative to $\alpha$ are the following sets $\Delta_i$ for all $i\in \{1,2,\dots,\lfloor m/2 \rfloor\}$:
\[ 
\Delta_i:= \left\lbrace \{ W_1,W_2\} \,\Big\vert \, \begin{pmatrix}
 \vert V_1 \cap W_1 \vert  &  \vert V_1 \cap W_2 \vert  \\
 \vert V_2 \cap W_1 \vert  &  \vert V_2 \cap W_2 \vert  \\
\end{pmatrix} =\begin{pmatrix}
i&m-i\\
m-i&i\\
\end{pmatrix} \text{ or } \begin{pmatrix} 
m-i&i\\
i&m-i\\
\end{pmatrix} \right\rbrace. 
\]
Consequently, $G$ has  $\lfloor m/2 \rfloor+1$ suborbits on $\Omega$.

Now we compute the length of $\Delta_i$ for every  $i\in \{1,2,\dots,\lfloor m/2 \rfloor\}$. 
Let $\beta=\{ W_1,W_2\} \in \Delta_i $.
Then $G_\alpha=\Sym(V_1) \times \Sym(V_2)\,{:}\,\langle g\rangle$ with $g$ an involution swapping $V_1$ and $V_2$, and $G_\beta=\Sym(W_1) \times \Sym(W_2)\,{:}\,\langle h\rangle$ with $h$ an involution swapping $W_1$ and $W_2$. 
Let $N=\Sym(V_1) \times \Sym(V_2)$ be the kernel of $G_\alpha$ acting on the partition $\alpha=\{ V_1,V_2\}$, and let $G_{\alpha\beta}=G_{\alpha} \cap G_{\beta}$.
It is easy to find one involution $x$ which swaps $V_2\cap W_1$ and $V_1\cap W_2$, and swaps $V_1 \cap W_1$ and $V_2 \cap W_2$.
This involution $x$ stabilizes the partitions $\alpha$ and $\beta$, and so $x\in G_{\alpha\beta}$.
Since $x$ swaps $V_1$ and $V_2$, we see that $x\notin N$ and hence 
\[
G_{\alpha\beta}/(N\cap G_{\alpha\beta })\cong  G_{\alpha\beta} N/N \cong \Sy_2,
\]
which implies $ \vert G_{\alpha\beta} \vert =2 \vert N\cap G_{\alpha\beta } \vert $.
Now for every $z\in N\cap G_{\alpha\beta }$, we have $V_1^z=V_1$ and  $V_2^z=V_2$, and either $(W_1,W_2)^z=(W_1,W_2)$ or $(W_1,W_2)^z=(W_2,W_1)$. 
Note that if $(W_1,W_2)^z=(W_2,W_1)$, then 
\[ i= \vert V_1\cap W_1 \vert = \vert (V_1\cap W_1)^z \vert = \vert V_1^z\cap W_1^z \vert = \vert V_1\cap W_2 \vert =m-i.\]

Suppose that $2i\neq m$. Then $i\neq m-i$, which implies that 
$(W_1,W_2)^z=(W_1,W_2)$ for all $z\in N\cap G_{\alpha\beta }$. Thus  $N\cap G_{\alpha\beta}=\Sym(V_1 \cap W_1)\times \Sym(V_1 \cap W_2) \times \Sym(V_2 \cap W_1)\times \Sym(V_2 \cap W_2)$, and so 
\[ G_{\alpha\beta}=(\Sym(V_1 \cap W_1)\times \Sym(V_1 \cap W_2) \times \Sym(V_2 \cap W_1)\times \Sym(V_2 \cap W_2))\,{:}\,\langle x\rangle. \]
Then
\[
 \vert \Delta_i \vert =\frac{ \vert G_\alpha \vert }{ \vert G_{\alpha\beta} \vert }=\frac{2(m!)^2}{2((m-i)!i!)^2}=\left(\frac{m!}{(m-i)!i!}\right)^2=2^{-\lfloor \frac{2i}{m}\rfloor}\left(\frac{m!}{(m-i)!i!}\right)^2.
\] 

Suppose that $2i=m$. Now $\vert \Omega \vert =2m=4i$ and $ \vert V_s\cap W_r \vert =i$ for all $s,r \in \{1,2\}$.
Write  $\Omega=\{t_1,t_2,\ldots,t_{4i} \}$ and
\[
\begin{array}{lll}
&V_1\cap W_1=\{t_1,t_{2}\ldots,t_i \},  &V_1\cap W_2=\{t_{i+1},t_{i+2},\ldots,t_{2i} \}, \\
&V_2\cap W_1=\{t_{2i+1},t_{2i+2},\ldots,t_{2i}\},& V_2\cap  W_2=\{t_{3i+1},t_{3i+2},\ldots,t_{4i} \}. 
\end{array} 
\]
Let $z=(t_1,t_{i+1})(t_2,t_{i+2})\cdots (t_i,t_{2i})(t_{2i+1},t_{3i+1})(t_{2i+2},t_{3i+2})\cdots (t_{3i},t_{4i})$.
Then $z$ swaps  $V_1 \cap W_1$ and $V_1 \cap W_2$, and  swaps $V_2 \cap W_1$ and $V_2 \cap W_2$.
Since $z$ fixes both $V_1$ and $V_2$, and swaps $W_1$ and $W_2$, it follows that $z \in N=\Sym(V_1) \times \Sym(V_2)$, and $z$ fixes both partitions $\alpha=\{ V_1,V_2\}$ and $\beta=\{ W_1,W_2\}$, that is, $z\in G_{\alpha\beta}$.
Therefore, $z\in N\cap G_{\alpha\beta}$.
Let $M=\Sym(W_1) \times \Sym(W_2)$. 
Since $z$ swaps $W_1$ and $W_2$, we deduce that $M\cap N\cap G_{\alpha\beta }$ has index $2$ in $N\cap G_{\alpha\beta }$.
Note that $M\cap N\cap G_{\alpha\beta }=M\cap N=\Sym(V_1 \cap W_1)\times \Sym(V_1 \cap W_2) \times \Sym(V_2 \cap W_1)\times \Sym(V_2 \cap W_2)$. 
Thus 
\[G_{\alpha\beta}=(\Sym(V_1 \cap W_1)\times \Sym(V_1 \cap W_2) \times \Sym(V_2 \cap W_1)\times \Sym(V_2 \cap W_2))\,{:}\,\langle z,x\rangle.\]
For every $s\in \{1,2\}$ and  every  $r \in \{1,2\}$, since
\[ 
\begin{split}
&(V_s \cap W_r)^{xz}=(V_{3-s} \cap W_{3-r})^{z}=V_{3-s} \cap W_{r}, \\
&(V_s \cap W_r)^{zx}=(V_{s} \cap W_{3-r})^{x}=V_{3-s} \cap W_{r},
\end{split}
\]
we conclude that $zx=xz$ and so $\langle x,z \rangle\cong \Sy_2^2$.
Then
\[
 \vert \Delta_{i} \vert = \vert \Delta_{m/2} \vert =\frac{ \vert G_\alpha \vert }{ \vert G_{\alpha\beta} \vert }=\frac{2(m!)^2}{4((m-i)!i!)^2}=2^{-\lfloor \frac{2i}{m}\rfloor}\left(\frac{m!}{(m-i)!i!}\right)^2.
\] 
 
Now we consider the action of $L$ on $\Omega$. 
Since $m\geq 3$ and $i\leq m/2$, we see that $m-i \geq 2$, which implies both $G_{\alpha}$ and $G_{\alpha\beta}$ contain an odd permutation. 
It follows that  $ \vert G_{\alpha} \vert / \vert L_{\alpha} \vert =2$ and $ \vert  G_{\alpha\beta} \vert / \vert  L_{\alpha\beta} \vert =2$, and hence $ \vert L_{\alpha} \vert / \vert L_{\alpha\beta} \vert = \vert G_{\alpha} \vert / \vert G_{\alpha\beta} \vert $.
This implies that both $G$ and $L$ have $\lfloor m/2 \rfloor+1$ suborbits on $\Omega$, and have the same  nontrivial subdegrees. 
\end{proof}

\section{Proof of Theorem~\ref{th:main}} 
In this section, we prove Theorem~\ref{th:main}. 
We make the following hypothesis throughout.

\begin{hypothesis}\label{hy:1}
Let $G$ be an almost simple group such that $\Soc(G)=\A_n$ with $n\geq 5$. 
Suppose that $\mathcal{D}=(\mathcal{P},\mathcal{B})$ is a nontrivial $G$-block-transitive $3$-$(v,k,1)$ design. 
Let  $\alpha $ be a point in $\mathcal{P}$ and let $B$ be a block in $\mathcal{B}$ such that $\alpha \in B$.  
\end{hypothesis}
  
Since there exists no $G$-block-transitive and $G$-point-imprimitive $3$-$(v,k,1)$ design by ~\cite[Corollary~2.3(a)]{MT2001}, we conclude that $G$ acts primitively on $\mathcal{P}$, and hence the group $G_\alpha$ is maximal in $G$.  
Note that for every $n\neq 6$,  the automorphism group of $ \A_n$ is $\Sy_n$, while  $\Aut(\A_6) \cong \PGaL_2(9)=\A_6.2^2$.
For convenience, we shall first deal  with $\A_5$ and $\A_6$.
For the case $n\geq 7$,  from the classification of maximal subgroups of $\A_n$ and $\Sy_n$ given by Liebeck, Praeger and Saxl~\cite{LPS1987}, we conclude that one of the following holds:
 
\begin{enumerate} [\rm (a)]
\item $G_{\alpha}=(\Sy_m \times \Sy_{n-m}) \cap G$, with $2m<n$ and  $m\geq 1$ (intransitive case);
\item $G_{\alpha}=(\Sy_m \wr \Sy_{\ell}) \cap G$,  with $n=m\ell$, and $m\geq 2 $ and $\ell\geq 2$ (imprimitive case);
\item $G_{\alpha}=\AGL_d(p) \cap G$,  with $n=p^d$  (affine case);
\item $G_{\alpha}= (T^m{:}(\Out(T)\times \Sy_\ell)) \cap G$, with $T$  a nonabelian simple group, $\ell \geq 2$ and $n= \vert T \vert ^{\ell-1}$ (diagonal case);
\item $G_{\alpha}=(\Sy_m \wr \Sy_\ell) \cap G$,  with $n=m^\ell$, $m \geq 5$ and $\ell \geq 2$ (wreath case);
\item $T \leq G_{\alpha} \leq \Aut(T)$,  with  $T \neq \A_n$ a nonabelian simple group, and $G_{v}$ acting primitively on $\{ 1,2,\dots,n\}$ (almost simple case).
\end{enumerate} 

%Note that $G_\alpha$ acts on $\{1,2,\dots,n\}$ intransitively, imprimitively or primitively for (a), (b) and (c)--(f), respectively. 
  
\subsection{The case $\Soc(G)=\A_5$ or $\A_6$}
%In this subsection, we deal with the case $\Soc(G)=\A_5$ or $\A_6$. 

\begin{lemma}
The socle of $G$ is not $\A_5$.
\end{lemma} 

\begin{proof}
Suppose for a contradiction that the socle of $G$ is $\A_5$. Then $G=\A_5$ or $\Sy_5$.
Note that $k>3$ as $\mathcal{D}$ is nontrivial. 
By~Lemma~\ref{lm:vk}, that is, $v\geq k^2-3k+4$, we conclude that $v\geq 8$. 
By~Atlas~\cite{Atlas}, we see that $(G,G_\alpha)=(\A_5.\mathcal{O},\Sy_{3}.\mathcal{O})$ with $\mathcal{O}\leq 2$, and  $v=10$. 
Again, from $v\geq k^2-3k+4$ we conclude that $k=4$.
By~Lemma~\ref{lm:brlambda2}(a), $ \vert \mathcal{B} \vert =30$. 
Suppose first that $G=\A_5$.  
Then $G_\alpha\cong \Sy_3$  and $G_{B}\cong \ZZ_2$. 
Note that $\A_5$ has only one non-conjugate subgroup isomorphic to $\Sy_{3}$, and has only one non-conjugate subgroup isomorphic to $ \ZZ_2$.  
Since $G$ acts transitively on $\mathcal{P}$,  we may let
\[
\begin{split}
G&=\langle   (2, 3, 5)(4, 7, 10)(6, 9, 8), (1, 2, 4)(3, 6, 7)(5, 8, 10)\rangle,\\
G_\alpha&=\langle       (2, 3, 5)(4, 7, 10)(6, 9, 8),     (2, 4)(3, 10)(5, 7)(6, 8)\rangle.
\end{split}
\]
(Computation in \magma~\cite{Magma} shows  that,  up to permutation equivalence, the above permutation group $G$ is the unique permutation representation of $\A_5$ of degree $10$.) 
Further, since $G$ acts transitively on $\mathcal{B}$, the group $H:=\langle  (2, 4)(3, 10)(5, 7)(6, 8)\rangle $ is the stabilizer of some block $C$ in $G$, that is, $ G_{C}=H$.
Then $G_{C}$ has six orbits on $\mathcal{P}$, namely,
\[
O_1:=\{1\},O_2:=\{9\},O_3:=\{2, 4\},O_4:=\{3, 10\},O_5:=\{5, 7\},O_6:=\{6, 8\}.
\]
Since $ \vert C \vert =k=4$, there are $10$ choices for $C$, namely $C_i$ for all $i \in \{1,2,\dots, 10\}$ as follows:
\begin{align*}
&C_1:=\{ 1,9,2,4\},C_2:=\{ 1,9,3, 10\},C_3:=\{ 1,9,5, 7\},C_4:=\{ 1,9,6,8\},\\
&C_5:=\{ 2,4,3, 10\}, C_6:=\{ 2,4,5, 7\},C_7:=\{ 2,4,6, 8\},C_8:=\{ 3, 10,5, 7\},\\
&C_9:=\{5, 7,6,8\},C_{10}:=\{ 3, 10,6, 8\}.
\end{align*} 
Computation in \magma~\cite{Magma} shows that 
\begin{itemize}
\item $ \vert (C_i)^G \vert <30$ for $i\in \{1,2,3,4\}$; 
\item $ \vert (C_5)^G \vert =30$, while $(C_5)^G$ contains $\{  1, 3, 6, 10\}$ and $\{  1, 6, 8, 10\}$;
\item $ \vert (C_6)^G \vert =30$, while $(C_6)^G$ contains $\{  1, 2, 6, 10\}$ and $\{  1, 6, 9, 10\}$;
\item $ \vert (C_7)^G \vert =30$, while $\{ 1,6,10\}$ is not contained in any block in $(C_7)^G$; 
\item $ \vert (C_8)^G \vert =30$, while $\{ 1,6,10\}$ is not contained in any block in $(C_8)^G$; 
\item $ \vert (C_9)^G \vert =30$, while $(C_9)^G$ contains $\{  1, 2, 6, 10\}$ and $\{  1, 6, 9, 10\}$; 
\item $ \vert (C_{10})^G \vert =30$, while $(C_{10})^G$ contains $\{  1, 3, 6, 10\}$ and $\{  1, 6, 8, 10\}$. 
\end{itemize}
Therefore, the above computation results imply  that the block $C \neq C_i$ for  any $i \in \{1,2,\dots,10\}$, a contradiction.
 
The case $G=\Sy_5$ is treated similarly, and computation in \magma~\cite{Magma} shows that the case $G=\Sy_5$ is also impossible. 
\end{proof}

 For the case $\A_6$, it turns out that there exists an example for $\mathcal{D}$. 
 %This design is a special case of a family of flag-transitive Steiner $3$-designs on $\PSL_2(q)$, which was  introduced in~\cite[Theorem~3]{K1985}(b) and~\cite[Main~Theorem]{K1985}(2). 
 Note that $\A_6 \cong \PSL_2(9)$ and $\Aut(\A_6) \cong \Sy_6:\mathrm{Z}_2$.

\begin{example}\label{ex}
Let $G  $ be a subgroup of $\Sy_{10}$ generated by permutations
\[
    (3, 6, 8, 5, 7, 10, 9, 4), 
    (1, 8, 2)(3, 4, 5)(6, 10, 7),
    (3, 7)(4, 6)(5, 10).
\]
Computation in {\sc Magma}~\cite{Magma} shows that $G\cong \Sy_6:\mathrm{Z}_2$. Let $\mathcal{D}_{3,10,4}=(\mathcal{P},\mathcal{B})$ where
\[
\begin{split}
\mathcal{P}=&\{1,2,\dots,10\},\\
\mathcal{B}=&\{      
    \{ 1, 5, 7, 6 \},
    \{ 1,   7, 8,10 \},
    \{ 3, 6, 8, 10 \},
    \{ 1, 3, 4, 10 \},
    \{ 1, 5, 9, 10 \},
    \{ 2, 6, 7, 8 \},
    \{ 1, 3, 5, 8 \}, \\&
    \{ 5, 6, 8, 9 \},
    \{ 2,  4, 7,10 \},
    \{ 4, 5, 7,8  \},
    \{ 1, 3, 6, 9 \},
    \{ 1, 4, 7, 9 \},
    \{ 3, 7, 8, 9 \},
    \{ 2, 5, 8, 10 \},\\&
    \{  1,2, 6, 10 \},
    \{ 2, 3, 4, 8 \}, 
    \{ 3, 2, 10, 9 \},
    \{ 4, 8, 9, 10 \},
    \{ 2, 3, 5, 6 \},
    \{ 3, 5, 7, 10 \},
    \{ 1, 4, 6 , 8\},\\&
    \{ 4, 5, 6, 10 \},
    \{ 2, 4, 9, 6 \},
    \{ 2, 5, 7, 9 \}, 
    \{ 1, 2, 3, 7 \},
    \{ 1, 2, 8, 9 \},
    \{ 1, 2, 4, 5 \},
    \{   3, 4,5,9 \},\\&
    \{ 6, 7, 9, 10 \},
    \{ 3, 4, 6, 7 \}
 \}.\\
\end{split}
\]  
It is a straightforward verification that $\mathcal{D}$ is a $3$-$(10,4,1)$-design.
Computation in {\sc Magma}~\cite{Magma} shows that
$G$ acts transitively on $\mathcal{B}$, and  $G_{\{1,5,6,7\}} $ is generated by
\[
(1, 5)(2, 4)(3, 8),
 (1, 6, 5)(2, 9, 4, 3, 10, 8),
 (2, 8)(3, 4)(6, 7),
\]
which implies that $G_{\{1,5,6,7\}} $ is transitive on $\{1,5,6,7\}$.
Therefore, $\mathcal{D}$ is $G$-flag-transitive.
Moreover, computation in {\sc Magma}~\cite{Magma} shows that  both $ \A_6.2_2\cong\PGL(2,9)$ and $ \A_6.2_3\cong\M_{10}$ acts flag-transitively on $\mathcal{D}$, and  $G$ is maximal in $\Sy_{10}$, which implies that $\Aut(\mathcal{D}_{3,10,4})=G=\mathrm{S}_6:\mathrm{Z}_2$. 
\end{example}

\begin{lemma}
Suppose that $\Soc(G)=\A_6$. Then $G=\PGL_2(9)$, $\M_{10}$ or $\mathrm{S}_6:\mathrm{Z}_2$, and $\mathcal{D}\cong \Aut(\mathcal{D}_{3,10,4})$.
\end{lemma}  
\begin{proof}  
%Note that $\Aut(\A_6) \cong \mathrm{S}_6:\mathrm{Z}_2$. 
By Atlas~\cite{Atlas}, the possibilities for $G$ are $ \A_6$, $ \A_6.2_1 \cong \Sy_6$, $ \A_6.2_2 \cong \PGL(2,9)$, $ \A_6.2_3 \cong \M_{10}$ and $\Aut(\A_6)\cong \mathrm{S}_6:\mathrm{Z}_2$. 
Since $ k>3$ and $v\geq k^2-3k+4$, we have $v\geq 8$.
 
\vspace{3pt}
\underline{$G=\A_6$.} Since $v\geq 8$,  by~Atlas~\cite{Atlas} we have $(G_\alpha,v)=(3^2\,{:}\,4,10)$ or $(\Sy_4,15)$. 

Suppose $(G_\alpha,v)=(3^2\,{:}\,4,10)$. Then $k=4$ as $v\geq k^2-3k+4$.
By Lemma~\ref{lm:brlambda2}(a) we have $ \vert \mathcal{B} \vert =30$, and so $ \vert G_{B} \vert =12$.
Computation in~\magma~\cite{Magma} shows that $G$ has two non-conjugate subgroups of order $12$, and both of them have two orbits on $\mathcal{P}$ with lengths $4$ and $6$.
This implies that $G_{B}$ is transitive on $B$ and hence $ \mathcal{D}$ is $G$-flag-transitive. 
From~\cite[p.208]{H2005}, we see that $G$ should be $3$-homogeneous on $\mathcal{P}$.
However, computation in~\magma~\cite{Magma} shows that $ \A_6$ is not $3$-homogeneous on a set of $10$ points, a contradiction.
Actually, computation in~\magma~\cite{Magma} shows that for those two subgroups of order $12$ in $\A_6$, their orbits of length $4$ admits setwise stabilizer of order $24$ in $\A_6$.

Suppose $(G_\alpha,v)=(\Sy_4,15)$. Then $k=4$ or $5$. However, by Lemma~\ref{lm:brlambda2}(c),
$v-2$ is divisible by $k-2$. Thus, both $k=4$ and $k=5$ are impossible.

\vspace{3pt}
\underline{$G=\mathrm{S}_6$.}   
Since $v\geq 8$,  by~Atlas~\cite{Atlas} we have $(G_\alpha,v)=(3^2\,{:}\,\D_8,10)$ or $(\Sy_4 \times 2,15)$.
This case is ruled out with  similar arguments as in  the case $G=\A_6$. 
Note that computation in~\magma~\cite{Magma} shows that the action of $\A_6.2_1$ on $10$ points is also not $3$-homogeneous.
  
\vspace{3pt}
\underline{$G=\PGL_2(9)$.}
Since $v\geq 8$,  by~Atlas~\cite{Atlas} we have $(G_\alpha,v)=( \D_{20},36)$, $(3^2\,{:}\,8,10)$ or $( \D_{16},45)$.

Suppose $(G_\alpha,v)=( \D_{20},36)$. 
Then $k\in \{4,5,6,7\}$ as  $ k>3$ and $v\geq k^2-3k+4$. 
Since $v-2$ is divisible by $k-2$ by Lemma~\ref{lm:brlambda2}(c), we deduce that $k=4$. 
Then $ \vert \mathcal{B} \vert =1785$ by  Lemma~\ref{lm:brlambda2}(a).
This is a contradiction because $ \vert G \vert =720$ is divisible by $ \vert \mathcal{B} \vert $.

Suppose $(G_\alpha,v)=(3^2\,{:}\,8,10)$. 
From  $ k>3$ and $v\geq k^2-3k+4$ we obtain $k=4$.
Then  $\vert B\vert =30$ and $\vert G_B \vert=24$.
Computation in~\magma~\cite{Magma} shows that $G$ has only one  conjugate classes of subgroup  of order $24$, and moreover, $G_B$  has two orbits of lengths $4$ and $6$ on $\mathcal{P}$.
This implies that $\mathcal{D}$ is flag-transitive, and a block is the $G_B$-orbit of length $4$.
We may identify $G=\PGL_2(9)$ with a subgroup of the group $\mathrm{S}_6:\mathrm{Z}_2$ constructed in Example~\ref{ex}, that is,
\[ G=\langle (1, 4)(2, 10)(3, 5)(6, 7)(8, 9),
    (2, 6, 10)(3, 8, 5)(4, 9, 7) \rangle.\]
Then we may take $G_B$ as the subgroup generated by permutations 
\[
  (1, 6, 7)(2, 8, 10)(3, 4, 9), 
    (1, 5, 7)(2, 9, 4)(3, 10, 8), 
    (2, 10)(3, 9)(4, 8)(5, 7).
\]
Now the $G_B$-orbit of length $4$  is $\{1,5,6,7 \}$.
Computation shows that $\mathcal{D}$ is exactly the design $\mathcal{D}_{3,10,4}$ in Example~\ref{ex}. 
 
Suppose $(G_\alpha,v)=(\D_{16},45)$. 
Then from  $k>3 $ and $v\geq k^2-3k+4$ we conclude that $k\in \{4,5,6,7,8\}$. 
However,  $v-2$ is not divisible by $k-2$  for any $k\in \{4,5,6,7,8\}$, contradicting  Lemma~\ref{lm:brlambda2}(c).

\vspace{3pt}
\underline{$G=\M_{10}$.} Since $v\geq 8$,  by Atlas~\cite{Atlas} we have $(G_\alpha,v)=( 5\,{:}\,4,36)$, $(3^2\,{:}\,8,10)$ or $( 8\,{:}\,2,45)$.
The arguments for this case are  similar to that  for the case $G=\PGL_2(9)$. 
%Note that computation in~\magma~\cite{Magma} shows that the action of $\A_6.2_3$ on $10$ points is not $3$-homogeneous.

\vspace{3pt}
\underline{$G=\mathrm{S}_6:\mathrm{Z}_2$.} From Atlas~\cite{Atlas} we see that $(G_\alpha,v)=( 10\,{:}\,4,36)$, $(3^2\,{:}\,[2^4],10)$ or $( [2^5],45)$.
The arguments are also  similar.  
\end{proof}

\subsection{The intransitive case}

In this subsection, we assume that $n\geq 7$ and deal  with the case (a), where $G_\alpha$ acts intransitively on $\{1,2,\dots,n\}$. 
In this case, we may identify $\mathcal{P}$ with  the set of $m$-subsets of $\{1,2,\dots,n\}$. 
Recall that the nontrivial subdegrees of $G$ are given in Lemma~\ref{lm:subdegree}.

\begin{lemma}
Suppose that $\Soc(G)=\A_n$ with $n\geq 7$. Then $G_\a$ is not  of intransitive case.
\end{lemma}

\begin{proof}
Suppose that $G_{\alpha}=(\Sy_m \times \Sy_{n-m}) \cap G$, with  $m\geq 1$ and $n\geq 2m+1$. 
If $m=1$, then $v=n$, $G_\alpha=\A_{n-1}$ or $\Sy_{n-1}$, and   $G$ acts $2$-transitively on $\mathcal{P}$, and the case is proved to be impossible by~\cite{K1985}.
Therefore, $m\geq 2$. 

Since $m\geq 2$, the subgroup $\Sy_m$ of $\Sy_m \times \Sy_{n-m}$ contains an odd permutation on $\{1,2,\dots,n\}$,  which implies that $(\Sy_m \times \Sy_{n-m}) \cap \A_n$ is of index $2$ in $\A_n$. 
Hence
\begin{equation}\label{eq:vintrans}
v=\frac{ \vert \Sy_n \vert }{ \vert \Sy_{m} \vert \cdot  \vert \Sy_{n-m} \vert }=\frac{n!}{m!(n-m)!}=\frac{n(n-1)\cdots (n-m+1)}{m!}=\binom{n}{m}.
\end{equation}

Let $g=(1,2,3) \in G$. 
Since $G$ acts faithfully on $\mathcal{P}$, it follows that there exists at least one point not fixed by $g$, and hence
$\langle g\rangle$ has an orbit of length $3$ on $\mathcal{P}$.
Note that now we identify $\mathcal{P}$ with the set of $m$-subsets of $\{1,2,\dots,n \}$. 
Let $\beta \in \mathcal{P}$ be a $m$-subset fixed by $g$.
If $m=2$, then  $\{1,2,3\} \subseteq \{1,2,\dots,n \} \setminus \beta$ and hence there are $\binom{n-3}{m}$ choices for $\beta$.
If $m\geq 3$, then either $\{1,2,3\} \subseteq \{1,2,\dots,n \} \setminus \beta$ or $\{1,2,3\} \subseteq  \beta $, and hence there are $\binom{n-3}{m-3}+\binom{n-3}{m}$ choices for $\beta$. 
It follows that
\[
 \vert \Fix_{\mathcal{P}}(\langle g\rangle) \vert =\binom{n-3}{m} \text{ if }m=2, \text{ and }  \vert \Fix_{\mathcal{P}}(\langle g\rangle) \vert =\binom{n-3}{m-3}+\binom{n-3}{m} \text{ if }m\geq 3.
\] 
In particular,
\begin{equation}\label{eq:fixg}
 \vert \Fix_{\mathcal{P}}(\langle g\rangle) \vert \geq\binom{n-3}{m}.
\end{equation}
Since $v\geq k^2-3k+4$ by Lemma~\ref{lm:vk}, we have $v-k\geq k^2-4k+4=(k-2)^2$. 
Then from Lemma~\ref{lm:FixHvk} we conclude that
\[ 
 \vert \Fix_{\mathcal{P}}(\langle g\rangle) \vert  \leq \frac{2(v-k)}{k-2}+k-2 =\frac{2(v-k)+(k-2)^2}{k-2}\leq \frac{3(v-k)}{k-2}< \frac{3v}{k-2}.
\]
This  implies that
\begin{equation}\label{eq:intransk}
k< \frac{3v}{ \vert \Fix_{\mathcal{P}}(\langle g\rangle) \vert }+2 \leq 
3\binom{n }{m}/\binom{n-3}{m}+2=\frac{3n(n-1)(n-2)}{ (n-m)(n-m-1)(n-m-2)}+2.
\end{equation} 
By Lemma~\ref{lm:subdegree}, $G$ has a nontrivial subdegree $d:= m(n-m)$. 
Then from Lemma~\ref{lm:vkGad}(b), we conclude that
\begin{equation}\label{eq:intransv2kd}
(v-2)^2<(v-1)(v-2)\leq k(k-1)(k-2)d(d-1)=k(k-1)(k-2)m^2(n-m)^2.
\end{equation}
  
Suppose that the pair $(m,n)$ satisfying one of the following: 
 \begin{itemize}
\item $m=2$ and $7\leq n \leq 31$;
\item $m=3$ and $7\leq n \leq 16$;
\item $m=4$ and $9\leq n \leq 15$;
\item $(5,11)$, $(5,12)$,  $(5,13)$, $(5,14)$, $(6,13)$,  $(6,14)$, $(7,15)$.
\end{itemize}
Recall that the nontrivial subdegrees of $G$ are given in Lemma~\ref{lm:subdegree}. 
For every pair  $(n,m)$ above and every $k$ satisfying~\eqref{eq:intransk} and $k>3$,  computation shows that Lemma~\ref{lm:vkGad} holds only if $n=8$, $m=3$ and $k=11$.
In this case, we have $v=56$ and $ \vert \mathcal{B} \vert =168$. 
Suppose that $G=\Sy_8$. 
We may let $G_\a$ be the stabilizer of subset $\{1,2,3\}$ in $G$. 
Computation in~\magma~\cite{Magma} shows that there are two non-conjugate subgroups of index $168$ in $G$, say $H_1$ and $H_2$, and the orbits of $H_1$ on $\mathcal{P} $ are $O_i$ for $1\leq i \leq 6$ with $ \vert O_1 \vert =1$, $ \vert O_2 \vert =5$, $ \vert O_3 \vert = \vert O_4 \vert = \vert O_5 \vert =10$ and $ \vert O_6 \vert =20$, and the orbits of $H_2$ on $\mathcal{P} $ are $Q_i$ for $1\leq j \leq 3$ with $ \vert Q_1 \vert =6$, $ \vert Q_2 \vert =20$ and $ \vert Q_3 \vert = 30$.
Since $k=11$, that is, a block consists of $11$ points, we derive that $G_{B}$ is conjugate to $H_{1}$, and 
$\mathcal{B}=(O_1\cup O_3)^{G}$, $(O_1\cup O_4)^{G}$ or $(O_1\cup O_5)^{G}$.
However, for every candidate for $\mathcal{B}$,  computation in~\magma~\cite{Magma} shows that 
there exists some $3$-subset of $\{1,2,\dots,v\}$ which is contained in at least two blocks, a contradiction.
The computation for the case $G=\A_8$ is similar, and it turns out that  the case $G=\A_8$ is still impossible.  
%
%Therefore, it remain to consider the following cases:
%\begin{itemize}
%\item $m=2$ and $ n \geq 129$;
%\item $m=3$ and $ n \geq 22$;
%\item $m=4$ and $ n \geq 14$;
%\item $m=5$ and $ n \geq 12$;
%\item $m\geq 6$.
%\end{itemize}

%Next, we divide the remaining candidates for pair $(n,m)$ into several cases.

 Suppose $m=2$ and  $n\geq 32$.  By~\eqref{eq:intransk} we see that
\[
k< \frac{3n(n-1)(n-2) }{ (n-2)(n-3)(n-4)}+2   < 3\left( \frac{n-1}{n-4}\right)^2+2 \leq  3\left( \frac{32-1}{32-4}\right)^2+2<6.
\]  
Thus $k\leq 5$. Then~\ref{eq:intransv2kd} is reduced to
\[
 \left( \frac{n(n-1)}{2} -2 \right)^2 < 5\cdot 4\cdot 3\cdot 2^2\cdot(n-2)^2. 
\]
Computation shows that the above inequality does not hold for any $n\geq 32$, a contradiction.

%With similar arguments, one can rule the following cases:
%\begin{itemize} 
%\item $m=3$ and $ n \geq 17$;
%\item $m=4$ and $ n \geq 16$;
%\item $m=5$ and $ n \geq 15$;
%\item $m =6$ and $ n \geq 15$;
%\item $m =7$ and $ n \geq 16$.
%\end{itemize}
%To avoid many repeated arguments, we omit the proof.  
 
  Suppose $m=3$ and $n\geq 17$.  
Then by~\eqref{eq:intransk} we have
\[
k< \frac{3n(n-1)(n-2) }{ (n-3)(n-4)(n-5)}+2 < 3\left( \frac{n-2}{n-5}\right)^3+2\leq  3\left( \frac{17-2}{17-5}\right)^3+2<8.
\]  
This implies $k\leq 7$.
Then~\eqref{eq:intransv2kd} is reduced to
\[
\left( \frac{n(n-1)(n-2)}{6} -2\right)^2 < 7\cdot 6\cdot 5 \cdot 3^2\cdot(n-3)^2. 
\]
 Computation shows that there is no $n\geq 17$ satisfying the above inequality.

To avoid many repeating arguments, we omit the proof for the following cases:
\begin{itemize}
\item $m=4$ and $n\geq 16$;
\item $m=5$ and $n\geq 15$;
\item $m=6$ and $n\geq 15$;
\item $m=7$ and $n\geq 16$.
\end{itemize} 

Finally, we suppose that  $m\geq 8$.  
Since $n\geq 2m+1$, we have $n\geq 17$, and $m\leq n-m-1$ and $m\leq (n-1)/2$.
Then by Lemma~\ref{lm:vkGad}(b), we conclude that
\[
\begin{split}
 (v-2)^2&< (v-1)(v-2)\leq k(k-1)(k-2)d(d-1)<(k-1)^3d^2\\
 &<\left(\frac{3n(n-1)(n-2)}{ (n-m)(n-m-1)(n-m-2)}+1\right)^3m^2(n-m)^2\\
 &< \left(\frac{4n(n-1)(n-2)}{ (n-m)(n-m-1)(n-m-2)}\right)^3(n-m-1)^2(n-m)^2\\
 &=\frac{2^6n^3(n-1)^3(n-2)^3}{(n-m)(n-m-1)(n-m-2)^3}\leq \frac{ 2^{6}n^3(n-1)^3(n-2)^3}{\frac{n+1}{2}\frac{n-1}{2}\left(\frac{n-3}{2}\right)^3}\\
 & =  \frac{2^{11}n^3(n-1)^2}{n-3}  \cdot \frac{(n-2)^3}{(n+1)(n-3)^2}.
\end{split}
\]
Since $n\geq 17$, one has 
\[
\begin{split}
&(n-2)^3-(n+1)(n-3)^2=-n^2 + 9n -17<0, \\
&(48n^2-2)^2(n-3)-2^{11}n^3(n-1)^2=256  n^5 - 2816  n^4 - 2240   n^3 + 576  n^2 + 4 n - 12>0.
\end{split}
\]
Then
\[
(v-2)^2<\frac{2^{11}n^3(n-1)^2}{n-3} \cdot \frac{(n-2)^3 }{(n+1)(n-3)^2}< \frac{2^{11}n^3(n-1)^2}{n-3}<(48n^2-2)^2,
\]
and hence $v<48n^2$, that is,
\begin{equation}\label{eq:intrans}
 \frac{n(n-1)\cdots (n-m+1)}{m!} <48n^2.
\end{equation} 
We shall show that~\eqref{eq:intrans} does not hold for any $m\geq 8$ and $n\geq 2m+1$. 
To do this, we let  
\[ 
f(n,m) = \frac{n(n-1)\cdots (n-m+1)}{m!\cdot n^2}=\frac{(n-1)\cdots (n-m+1)}{m!n}.  
\]  
Since $n\geq 2m+1$, we have $m\leq (n-1)/2 $ and so
\[
n^2-(n+1)(n-m+1)\geq n^2-(n+1)\left(n-\frac{n-1}{2}+1 \right) =\frac{n^2-4n-3}{2}>0.
\]
Therefore,
\[ 
\frac{f(n+1,m)}{f(n,m)}=\frac{n}{n-m+1}\cdot\frac{n}{n+1}=\frac{n^2}{(n+1)(n-m+1)}>1, 
\]
which implies $f(n,m)\geq f(2m+1,m)$ for all $n\geq 2m+1$.
Let 
\[
h(m)=f(2m+1,m)=\frac{2m(2m-1)\cdots(m+2)}{m!(2m+1)}.
\]
Then
\[
\frac{h(m+1)}{h(m)}=\frac{(2m+2)(2m+1)}{m+2}\cdot\frac{1}{m+1}\cdot\frac{2m+1}{2m+3}= \frac{2(2m+1)^2}{(m+2)(2m+3)}>1.
\]
Consequently, $h(m)\geq h(8)  $ for all $m\geq 8$, and hence $f(n,m)\geq h(m)\geq h(8)$ for all $n\geq 2m+1$ and all $m\geq 8$.
Computation shows that $h(8)>84$.
Therefore, $f(n,m)>48$ for all $n\geq 2m+1$ and all $m\geq 8$,  contradicting~\eqref{eq:intrans}. 
\end{proof}

\subsection{The imprimitive case} 
In this subsection, we assume that $n\geq 7$ and deal  with the case (b), where $G_\alpha$ acts imprimitively on $\{1,2,\dots,n\}$. 
In this case, we may identify $\mathcal{P}$ with  the set of partitions of $\{1,2,\dots, m\ell\}$ with $\ell\geq 2$ blocks of size $m\geq 2$.  
Recall that in the case  $\ell=2$, the nontrivial subdegrees of $G$  are given in Lemma~\ref{lm:subdegree}.

\begin{lemma}
Suppose that $\Soc(G)=\A_n$ with $n\geq 7$. Then $G_\a$ is not of imprimitive case.
\end{lemma}

\begin{proof}
Suppose that $G_{\alpha}=(\Sy_m \wr \Sy_{\ell}) \cap G$,  with $n=m\ell$, and $m\geq 2 $ and $\ell\geq 2$.  Then
\begin{equation}\label{eq:imprimv}
v=\frac{ \vert \Sy_{m\ell} \vert }{ \vert \Sy_m \vert ^\ell  \vert  \Sy_{\ell} \vert }=\frac{(m\ell)!}{m!^\ell\ell!}
\end{equation}

Let $g=(1,2,3) \in G$.  
Since $G$ acts faithfully on $\mathcal{P}$, it follows that 
$\langle g\rangle$ has an orbit of length $3$ on $\mathcal{P}$.
Note that $\Cen_{\Sy_n}(g)=\langle g\rangle \times \Sym(\{4,5,\dots,n\})$.
Since $n\geq 7$, the group $ \Sym(\{4,5,\dots,n\})$ contains odd permutations, and hence  $ \vert \Cen_{\Sy_n}(g) \vert / \vert \Cen_{\A_n}(g) \vert =2$. 
Therefore, $ \vert G \vert / \vert \Cen_G(g) \vert =n(n-1)(n-2)/3$. 
This together with Lemma~\ref{lm:vkGCGg} and Lemma~\ref{lm:aninequation} prove the following   
\begin{equation}\label{eq:vn6/9}
v\leq \left(\frac{n(n-1)(n-2)}{3}+2 \right)^2< \frac{n^6}{9}. 
\end{equation} 
Consequently, we have
\begin{equation}\label{eq:imprimv<n6/9}
\frac{(m\ell)!}{ m!^\ell\ell!} < \frac{ (m\ell)^{6} }{9}.
\end{equation}

Suppose first that $m\leq 20$ and $\ell \leq 20$. 
Then computation shows that the pairs $(m,\ell)$ satisfying~\eqref{eq:imprimv<n6/9} are 
\begin{itemize}
\item[\rm (i)] $\ell=2$ and $4\leq m\leq 15$ (noting that $n=m\ell\geq 7$); 
\item[\rm (ii)]  $(3, 3)$, $(4, 3)$, $(5, 3)$, $(6, 3)$, $(2,4)$, $(3,4)$, $(2,5)$, $(2,6)$, $(2,7)$. 
\end{itemize} 
For those  $(m,\ell)$ satisfying (i), the nontrivial subdegrees of $G$ are determined by Lemma~\ref{lm:subdegree-imprimitive}, and for those $(m,\ell)$ satisfying (ii),  the nontrivial subdegrees of $G$ can be obtained by computation in \magma~\cite{Magma}.
However, computation shows that for those $(m,\ell)$ satisfying (i) or (ii), there exists no integer $k$ with $3< k\leq \sqrt{v}+2$ such that  Lemma~\ref{lm:vkGad} holds.

Therefore, $(m,\ell)$ is not a pair such that $m\leq 20$ and $\ell \leq 20$. Let 
\[
f(m,\ell)=\frac{(m\ell)!}{m!^\ell (\ell!)}/\frac{ (m\ell)^{6} }{9}=\frac{9(m\ell)!}{m!^\ell \ell! (m\ell)^{6} }.
\] 
Then
\[
\begin{split}
\frac{f(m+1,\ell)}{f(m,\ell)}&=\frac{((m+1)\ell)!}{(m\ell)!}\frac{m!^\ell }{(m+1)!^{\ell }}\frac{(m\ell)^{6}}{(m+1)^{6}\ell^6}=\frac{((m+1)\ell)!}{(m\ell)!}\cdot \frac{1}{(m+1)^\ell}\cdot \frac{m^{6}}{(m+1)^{6}},
\\
\frac{f(m,\ell+1)}{f(m,\ell)}&=\frac{(m(\ell+1))!}{(m\ell)!}\frac{m!^\ell }{m!^{\ell+1}}\frac{\ell!}{(\ell+1)!}\frac{(m\ell)^{6}}{m^{6}(\ell+1)^6}=\frac{(m(\ell+1))!}{(m\ell)!} \cdot \frac{1}{m! } \cdot \frac{\ell^6}{(\ell+1)^7}.
\end{split} 
\]
Since
\[ 
\frac{((m+1)\ell)!}{(m\ell)!} =\prod_{i=1}^{\ell}(m\ell+i)> (m\ell)^\ell, 
\]
we have 
\[ 
 \frac{f(m+1,\ell)}{f(m,\ell)} > \frac{(m\ell)^\ell m^{6 }}{(m+1)^{6+\ell}}=\ \ell^\ell\left(\frac{m}{m+1}\right)^{6+\ell }=\left(\frac{\ell^{\frac{\ell}{6+\ell}} m}{m+1}\right)^{6+\ell}.
\]

Suppose that $2\leq \ell \leq 20$. Then  $m\geq 21$ and so
\[
\frac{f(m+1,\ell)}{f(m,\ell)} \geq \left(\frac{\ell^{\frac{\ell}{6+\ell}} m}{m+1}\right)^{6+\ell }   \geq \left(\frac{2^{\frac{2}{6+2}} 21}{21+1}\right)^{6+\ell }   > 1^{6+\ell }=1,
\] 
which implies that if $\ell$ is fixed then  $f(m,\ell)$ is increasing as a funtion of $m$
 where $m\geq 21$ and $2\leq \ell \leq 20$, i.e. $f(m,\ell)\geq f(21,\ell)$. 
Since
\[
\begin{split}
\frac{f(21,\ell+1)}{f(21,\ell)}&=\left(\prod_{i=1}^{21}(21\ell+i) \right)\frac{\ell^6}{21!(\ell+1)^7}=\left(\prod_{i=1}^{20}\frac{21\ell+i}{i}\right)\frac{(21\ell+21)\ell^6}{21(\ell+1)^7} \\
&=\left(\prod_{i=1}^{20}\frac{21\ell+i}{i}\right) \left(\frac{ \ell}{\ell+1}\right)^6
 >42\left(\frac{2}{2+1}\right)^6>1,
 \end{split}
\]
we have $f(21,\ell)\geq f(21,2)>1$ for all $2\leq \ell \leq 20$, and so $f(m,\ell) >1$ for all $m\geq 21$ and $2\leq \ell \leq 20$, contradicting~\eqref{eq:imprimv<n6/9}. 

Suppose $\ell\geq 21$. Then $m\geq 2$ and
\[
\frac{f(m+1,\ell)}{f(m,\ell)} \geq \left(\frac{\ell^{\frac{\ell}{6+\ell}} m}{m+1}\right)^{6+\ell }   \geq \left(\frac{21^{\frac{21}{6+21}} 2}{2+1}\right)^{6+\ell }   > 1^{6+\ell }=1,
\]
which implies that $f(m,\ell)\geq f(2,\ell)$ for all $m\geq 2$ and $\ell \geq 21$.
Since
\[
\frac{f(2,\ell+1)}{f(2,\ell)}= \frac{(2\ell+2)(2\ell+1)\ell^6}{2(\ell+1)^7}= (2\ell+1)\left(\frac{ \ell}{\ell+1}\right)^{6 }   \geq (2\cdot 21+1)\left(\frac{ 21}{21+1}\right)^{6 }>1, 
\]
we have $f(2,\ell)\geq f(2,21)>1$ for all $\ell \geq 21$ and so $f(m,\ell) >1$ for all $m\geq 2$ and $\ell \geq 21$, contradicting~\eqref{eq:imprimv<n6/9}. 
\end{proof}

\subsection{The primitive case} 
In this subsection, we assume that $n\geq 7$ and deal with the cases (c)--(f) in the beginning of Section 3, where $G_\alpha$ acts  primitively on $\{1,2,\dots,n\}$ and $\A_n \not\leq G_\alpha $.

%We first note an important observation. 
%According to~\cite[Theorem~3.3B]{DM-book}, if  $G_\alpha$ contains a permutation $g$ fixing more than $n/2$ points on $\{ 1,2,\dots,n\}$, then $G_\alpha \geq \A_n$, which is not the case. 

\begin{lemma}\label{lm:Anprimcase}
Suppose that $\Soc(G)=\A_n$ with $n\geq 7$.  Then $G_\alpha$ is not of primitive affine,  diagonal,  wreath product  and  almost  case.
\end{lemma}

\begin{proof}
Suppose for a contradiction that $G_\alpha$ is one of the cases (c)--(f).
Then $G_\a$ acts primitively on $\{1,2,\dots,n\}$ and $\A_n \not\leq G_\alpha $.
Let $g=(1,2,3)$.
If $g \in G_\beta$ for some $\beta \in \mathcal{P}$, then we  conclude from~\cite[Theorem~3.3A]{DM-book} that $G_\beta \geq \A_n$, a contradiction.
Therefore, $g $ fixes no point in $\mathcal{P}$.
Then $k$ divides $v$ by Lemma~\ref{lm:gfixeno}. 

Suppose that $v$ is odd. From~\cite[Theorem~C]{K1987} we conclude that one of the following holds:
\begin{itemize}
\item $G=\A_7$ and $G_\alpha=\PSL_3(2)$;
\item $G=\A_8$, and $G_\alpha=2^3{:}\SL_3(2)\cong\AGL_3(2)$.
\end{itemize} 
Both of these two cases lead to $v=15$. 
Since $k\geq 4$ and $k$ divides $v$, we obtain $k=5$, which contradicts that $v-2$ is divisible by $k-2$ (see Lemma~\ref{lm:brlambda2}(c)). 
 
Suppose that $v$ is twice an odd integer. 
From Atlas~\cite{Atlas} we conclude $n\neq 7$ or $8$, and hence $n\geq 9$.
Note that $\Sym(\{1,2,\dots,8\})$ has a Sylow $2$-subgroup $P$ generating by the following permutations 
\[ 
(1,2),(3,4),(5,6),(7,8),(1,3)(2,4),(5,6)(7,8),(1,5)(2,6)(3,7)(4,8).
\]
By the Sylow theorems, we can take a Sylow $2$-subgroup $Q$ of $G $ such that $Q \geq P \cap G$. 
Again by the Sylow theorems, together with the transitivity of $G$ on $\mathcal{P}$, we see that there exists some $\beta \in \mathcal{P}$ such that a Sylow $2$-subgroup $R$ of $G_{\beta}$ is contained in $Q $. 
Since $ \vert G \vert / \vert G_\beta \vert =v$, which is twice an odd integer, it follows that $R$ is of index $2$ in $Q$. 
Let $g_1=(1,2)(3,4)$ and $g_2=(1,2)(5,6)$.  
Clearly, both $g_1$ and $g_2 $ are in $P\cap G$ and hence in $Q$.
Note that now $ G_\beta$ acts primitively on $\{1,2,\dots,n\}$.
According to~\cite[Theorem~3.3B]{DM-book}, if $ G_\beta$ contains one of $g_1$ and $g_2$, then $G_\beta \geq \A_n$, a contradiction.
Therefore neither $g_1$ nor $g_2$ is in $G_\beta$.
Now it follows from  $g_1,g_2 \notin R$ and $ \vert Q \vert / \vert R \vert =2$ that  $g_1g_2=(3,4)(5,6) \in R$. 
However, applying~\cite[Theorem~3.3B]{DM-book} again we derive that $G_\alpha \geq \A_n$, a contradiction.

Therefore, $v$ is divisible by $4$.
Then it follows from Lemma~\ref{lm:vdivisiblebyk4} that $\mathcal{D}$ is $G$-flag-transitive.
However, this contradicts the result of Huber~\cite{H2005}. 
\end{proof}

\begin{remark} 
There is an alternative way to prove Lemma~\ref{lm:Anprimcase}. 
It is shown in~\cite[Theorem 1.1]{M2002} that the order of a primitive group of degree $n$ is no more than $n^{1+\lfloor\log_2(n)\rfloor}$ apart from a few exceptions.
Using this upper bound and~\eqref{eq:vn6/9} we  conclude
\[
\frac{n^6}{9}>v>\frac{ \vert \A_m \vert }{ \vert G_\a \vert }\geq \frac{n!}{2n^{1+\lfloor\log_2(n)\rfloor}}\geq \frac{n!}{2n^{ \log_2(n) }}.
\]
Computation shows that the above inequality holds only if $n\leq 13$, and hence we only need to investigate primitive groups of degree at most $13$.   
\end{remark}

\section*{Declarations}

%Some journals require declarations to be submitted in a standardised format. Please check the Instructions for Authors of the journal to which you are submitting to see if you need to complete this section. If yes, your manuscript must contain the following sections under the heading `Declarations':
 
\begin{itemize}
%\item 
\item  This work was supported by the National Natural Science Foundation of China (12071484,12271524,12071023,11971054).The authors are very grateful for the anonymous referees' valuable comments to improve the paper. 
%Fu-Gang Yin is 
\item    
The authors declare they have no financial interests.
%\item Ethics approval 
%\item Consent for publication
\item Availability of data and materials. Data sharing not applicable                                                                                                                                                                                                                                                                                                                                                                                                                                                                                                                                                                                                                                                                                                                                                                                                                                                                                                                                                                                                                                                                                                                                                                                                                                                                                                                                                                                                                                                                                                                                                                                                                                                                                                                                                                                                                                                                                                                                                                                                                                                                                                                                                                                                                                                                                                                                                                                                                                                                                                                                                                                                                                                                                                                                                                                                                                                                                                                                                                                                                                                                                                                                                                                                                                                                                                                                                                                                                                                                                                                                                                                                                                                                                                                                                                                                                                                                                                                                                                                                                                                                                                                                                                                to this article as no datasets were generated or analysed during the current study.
%\item Authors' contributions
\end{itemize}

\end{document}